\documentclass[11pt,reqno]{amsart}

\usepackage[utf8]{inputenc}
\usepackage[T1]{fontenc}
\usepackage[svgnames,hyperref]{xcolor}
\usepackage{lmodern,amssymb}
\usepackage[UKenglish,USenglish,english,american]{babel}
\usepackage[cmtip]{xy}
\xyoption{pdf}
\xyoption{color}
\xyoption{all}

\usepackage[shortlabels]{enumitem}
\setlist[enumerate]{label=\rm{(\arabic*)}}
\setlist[enumerate,2]{label=\rm({\it\roman*})}
\setlist[itemize]{label=\raisebox{0.25ex}{\tiny$\bullet$}}

\theoremstyle{plain}    
 \newtheorem{thm}{Theorem}[section]
 \numberwithin{equation}{section} %% Comment out for sequentially-numbered
 \numberwithin{figure}{section} %% Comment out for sequentially-numbered
 \theoremstyle{plain}
 \theoremstyle{plain}    
 \newtheorem{cor}[thm]{Corollary} %%Delete [thm] to re-start numbering
 \theoremstyle{plain}    
 \newtheorem{prop}[thm]{Proposition} %%Delete [thm] to re-start numbering
 \theoremstyle{plain}    
 \newtheorem{lem}[thm]{Lemma} %%Delete [thm] to re-start numbering
 \theoremstyle{remark}
 \newtheorem{rmk}[thm]{Remark}
 \theoremstyle{definition}

\theoremstyle{definition}
\newtheorem{defi}[thm]{Definition}

\newcommand{\C}{{\mathbb{C}}}
\newcommand{\N}{{\mathbb{N}}}

\newcommand{\R}{{\mathbb{R}}}

\newcommand{\e}{\varepsilon}

\newcommand{\f}{\varphi}

\DeclareMathOperator{\tr}{tr}

\newcommand{\mycolor}{Navy}
\usepackage[pdfauthor={Dat T. T\^O}, colorlinks, linktocpage, citecolor = \mycolor, linkcolor = \mycolor, urlcolor = \mycolor]{hyperref}
\usepackage[all]{hypcap} % needed to help hyperlinks direct correctly;

\title{Regularizing properties of Complex Monge-Amp\`ere flows}
\author[T. D. T\^O]{Tat Dat T\^O}
\date{\today}
\address{ Institut Math\'ematiques de Toulouse, \\
Universit{\'e} Paul Sabatier\\ 31062 Toulouse cedex 09\\ France}

\email{tat-dat.to@math.univ-toulouse.fr}

\begin{document}
\begin{abstract}
We study the regularizing properties of complex Monge-Amp\`ere flows on a K\"ahler manifold $(X,\omega)$ when the initial data are $\omega$-psh functions  with zero Lelong number at all points. We  prove that the general Monge-Amp\`ere flow has a solution which is immediately smooth.  We also prove the uniqueness and stability of solution.
\end{abstract}
\maketitle
%\tableofcontents

\section*{Introduction}
 Let $(X,\omega)$ be a compact K\"ahler manifold of complex dimension $n$ and $\alpha\in H^{1,1}(X,\R)$ a K\"ahler class with $\omega\in \alpha$. Let $\Omega$ be a smooth volume form on $X$. Denote by $(\theta_t)_{t\in [0,T]}$ a family of K\"ahler forms on $X$, and assume that $\theta_0=\omega$. The goal of this note is to prove the regularizing and stability properties of solutions to the following complex Monge-Amp\`ere flow
\begin{equation*}
(CMAF) \hskip1cm 
\dfrac{\partial \f_t}{\partial t}= \log\dfrac{(\theta_t+dd^c\f_t)^n}{\Omega}-F(t,z,\f_t)  
\end{equation*} 
where $F$ is a smooth function and $\f(0,z)=\f_0(z)$ is a $\omega$-plurisubharmonic ($\omega$-psh) function  with zero Lelong numbers at all points. 

\medskip

One motivation for studying this Monge-Amp\`ere flow is  that the K\"aler-Ricci flow can be reduced to a particular case of $(CMAF)$. When $F=F(z)$ and $\theta_t=\omega+t\chi$, where $\chi=\eta-Ric(\omega)$, then $(CMAF)$ is the local potential equation of the twisted K\"ahler-Ricci flow
 \begin{equation*}
 \frac{\partial\omega_t }{\partial t}= -Ric(\omega_t)+\eta,
 \end{equation*}
which was studied recently by Collins-Sz\'ekelydihi \cite{CSz12} and Guedj-Zeriahi \cite{GZ13}.

\medskip Running the K\"ahler-Ricci flow from a rough initial data has been the purpose of several recent works \cite{CD07}, \cite{ST}, \cite{SzTo}, \cite{GZ13}, \cite{BG13}, \cite{DiNL14}. In \cite{ST}, \cite{SzTo} the authors succeeded to run $(CMAF)$ from continuous initial data, while \cite{DiNL14} and \cite{GZ13} are running a simplified flow starting from an initial current with zero Lelong numbers. In this note we extend these latter works  to deal with general $(CMAF)$ and arbitrary initial data.

\medskip
A strong motivation for studying $(CMAF)$ with degenerate initial data comes from the Analytic Minimal Model Program introduced by J. Song and G. Tian \cite{ST}, \cite{ST12}.  It requires to study the behavior of the K\"ahler-Ricci flow on mildly singular varieties, and one is naturally lead to study weak solutions of degenerate complex Monge-Amp\`ere flows (when the function $F$ in $(CMAF)$ is not smooth but continuous). Eyssidieux-Guedj-Zeriahi have developed in \cite{EGZ14} a viscosity theory for degenerate complex Monge-Amp\`ere flows which allows in particular to define and study the K\"ahler-Ricci flow on varieties with canonical singularities.

\medskip
Our main result is the following:

\medskip
\textbf{Theorem A.}
\textit{Let $\f_0$ be a $\omega$-psh function  with zero Lelong numbers at all points. Let $(t,z,s)\mapsto F(t,z,s)$ be a smooth function on $[0,T]\times X\times \R$ such that $\frac{\partial F}{\partial s}\geq -C$ for some $C\geq 0$.\\
 Then there exists a family of smooth strictly $\theta_t-psh$ functions $(\varphi_t)$ satisfying $(CMAF)$
in $(0, T]\times X,$ with $\varphi_t\rightarrow \varphi_0$ in $L^1(X),$ as $t\searrow 0^+$. This family is moreover unique if $C=0$ and $|\frac{\partial F}{\partial t}|<C'$ for some $C'>0$.}

\medskip
We further show that 

\begin{itemize}
\item $\f_t$ converges to $\f_0$ in $C^0(X)$ if $\f_0$ is continuous.
\item $\varphi_t$ converges to $\varphi_0$ in capacity if $\f_0$ is merely bounded.
\item $\varphi_t$ converges to $\varphi_0$ in energy if $\varphi\in {\mathcal E}^1(X,\omega)$ has finite energy.
\end{itemize}
Moreover, we also prove  the following stability result:

\medskip
\textbf{Theorem B.} \textit{Let $\varphi_0,\varphi_{0,j}$ be $\omega$-psh functions with zero Lelong number at all points, such that $\f_{0,j}\rightarrow \f_0$ in $L^1(X)$. Denote by $\varphi_{t,j}$ and $\f_j$ the corresponding solutions of $(CMAF)$ with initial condition $\f_{0,j}$ and $\f_0$ respectively. Then for each $\e\in(0,T)$
$$\f_{t,j}\rightarrow \f_{t}\  \text{ in }\  C^\infty ([\e, T]\times X)\ \text{ as }\  j\rightarrow +\infty.$$
\quad Moreover, if $\f_0$ and $\psi_0$ are continuous, then for any $k\geq 0$, for any $0<\e<T$, there exists a positive constant $C(k,\e)$ depending only on $k$ and $\e$ such that
\begin{equation*}
||\f-\psi||_{C^{k}([\e,T]\times X)}\leq C(k,\e)||\f_0-\psi_0||_{L^\infty(X,\omega)
}.
\end{equation*}}
We also prove in Section \ref{nef} that one can run the Monge-Amp\`ere flow from a positive current representing a nef class, generalizing results from \cite{GZ13}, \cite{DiNL14}.

\medskip
The paper is organized as follows. In Section \ref{strategy} we recall some analytic tools, and give the strategy of proof of Theorem A. In Section \ref{a priori estm} we prove  various a priori estimates for the regular case. In Section \ref{proof} we prove Theorem A using the a priori estimates from Section \ref{a priori estm}.  In Section \ref{uniqueness and stability} we prove the uniqueness in Theorem A and Theorem B. In Section \ref{nef} we show that the Monge-Amp\`ere flow can run from a nef class.

\medskip
\textbf{Acknowledgement.} The author is grateful to his supervisor Vincent Guedj for support, suggestions and encouragement. We also would like to thank  Hoang Son Do, Eleonora Di Nezza, Hoang Chinh Lu, Van Hoang Nguyen and Ahmed Zeriahi  for very useful discussions. 
\section{Preliminaries and Strategy}\label{strategy}

In this section we recall some analytic tools which will be used in the sequel.
\subsection{Plurisubharmonic functions and Lelong number}
Let $(X,\omega)$ be a compact K\"ahler manifold. We define the following operators:
 $$d:=\partial +\bar{\partial},\quad d^c:= \frac{1}{2i\pi}(\partial-\bar{\partial}).$$
\begin{defi}
We let $PSH(X,\omega)$ denote the set of all {\sl$\omega$-plurisubharmonic functions} ($\omega$-psh for short), i.e the set of functions $\f\in L^1(X,\R\cup \{-\infty\})$ which can be locally written as the sum of a smooth and a plurisubharmonic function, and such that$$ \omega+dd^c\f\geq 0$$
in the weak sense of positive currents. 
\end{defi}

\begin{defi}
Let $\f$ be a $\omega$-psh function and  $x\in X$. The {\sl Lelong number} of $\f$ at $x$ is
$$\nu(\f,x):=\liminf_{z\rightarrow x} \frac{\f(z)}{\log |z-x|}.$$
We say $\f$ has a {\sl logarithmic pole of coefficient $\gamma$} at $x$ if $\nu(\f,x)=\gamma$. 
\end{defi}

\subsection{A Laplacian inequality }
Let $\alpha$ and $\omega$ be $(1,1)$-forms on a complex manifold $X$ with $\omega>0$. Then the trace of $\alpha$ with respect $\omega$ is defined as 
$$\tr_\omega(\alpha)=n\frac{\alpha\wedge \omega^{n-1}}{\omega^n}.$$  We can diagonalize $\alpha$ with respect to $\omega$ at each point of $X$, with real eigenvalues $\lambda_1\,\ldots,\lambda_n$ then $\tr_\omega(\alpha)=\sum_j\lambda_j$. 
The Laplace of a function $\varphi$ with respect to $\omega$ is given by $$\Delta_\omega\varphi =\tr_\omega(dd^c\varphi).$$
We have the following eigenvalue estimate:
\begin{lem}\label{Trace ineq}
If $\omega$ and $\omega'$ are two positive $(1,1)$-forms on a complex manifold $X$ of dimension $n$, then
$$
\left( \frac{\omega'^n}{\omega^n}\right)^{\frac{1}{n}}\leq \frac{1}{n}\tr_\omega(\omega')\leq \left( \frac{\omega'^n}{\omega^n} \right) (\tr_{\omega'}(\omega))^{n-1}.
$$
\end{lem}
The next result is a basic tool for establishing second order a priori estimates for complex Monge-Amp\`ere equations. 
\begin{prop}[\cite{Siu}]
\label{Laplace ineq}
Let $\omega,\omega'$ be two K\"aler forms on a compact complex manifold. If the holomorphic bisectional curvature of $\omega$ is bounded below by a constant $B\in \R$ on $X$, then we have
$$
\Delta_{\omega'}\log\tr_\omega(\omega')\geq -\frac{\tr_\omega Ric(\omega')}{\tr_\omega(\omega')}+B\tr_{\omega'}(\omega).$$
\end{prop}
\subsection{Maximum principle and comparison theorem}
We establish here a slight generalization of the comparison theorem that we will need.
\begin{prop}\label{comparison}
Let $\f, \psi\in C^\infty ([0,T]\times X)$ be $\theta_t$-psh functions such that
$$\frac{\partial \f}{\partial t}\leq \log\frac{\left(\theta_t+dd^c \f\right)^n}{\Omega}-F(t,z,\f),$$
$$\frac{\partial \psi}{\partial t}\geq \log\frac{\left(\theta_t+dd^c \psi\right)^n}{\Omega}-F(t,z,\psi),$$
where $F(t,z,s)$ is a smooth function with $\frac{\partial F}{\partial s}\geq -\lambda$. Then 
\begin{equation}\label{compa-max}
\sup_{[0,T]\times X}(\f_t-\psi_t)\leq e^{\lambda T}\max\, \bigg\{\sup_X(\f_0-\psi_0);0\bigg\}.
\end{equation}  
In particular, if  $\varphi_0\leq \psi_0$, then $\varphi_t\leq \psi_t$.
\end{prop}
\begin{proof} 
We define $u(x,t)=e^{-\lambda t}(\f_t-\psi_t)(x)-\e t\in C^\infty([0,T]\times X)$ where $\e>0$ is fixed. Suppose $u$ is maximal at $(t_0,x_0)\in [0,T]\times X$. If $t_0=0$ then we have directly the estimate (\ref{compa-max}). Assume now $t_0>0$, using the maximum principle, we get $\dot{u}\geq 0$ and  $dd^c_x u\leq 0 $ at $(t_0,x_0)$, hence
$$ -\lambda e^{-\lambda t}(\f_t-\psi_t)+e^{-\lambda t}(\dot{\f_t}-\dot{\psi_t}) \geq \e>0 \text{ and } dd^c_x \f_t\leq dd^c_x\psi_t.$$
Observing that at $(t_0,x_0)$
$$\dot{\f}-\dot{\psi}\leq F(t,x,\psi)-F(t,x,\f),$$  we infer that
$$0< F(t,x,\psi)+ \lambda \psi -[F(t,x,\f)+\lambda \f],$$
at $(t_0,z_0)$.
Since $\frac{\partial F}{\partial s}\geq -\lambda$, $F(t,x,s)+\lambda s$ is an increasing function in $s$, hence $\f_{t_0}(x_0)\leq \psi_{t_0}(x_0)$. Thus $u(x,t)\leq u(x_0,t_0)\leq 0$. Letting $\e\rightarrow 0$, this yields 
$$\sup_{[0,T]\times X}(\f_t-\psi_t)\leq e^{\lambda T}\max\,\bigg\{\sup_X(\f_0-\psi_0); 0\bigg\}.$$
\end{proof}
The following proposition has been given for the twisted K\"ahler-Ricci flow by Di Nezza and Lu \cite{DiNL14}:
\begin{prop}\label{weak comparison}
Assume $\psi_t$ a smooth solution of $(CMAF)$ with a smooth initial data $\psi_0$ and $\varphi_t$ is a subsolution of $(CMAF)$ with initial data $\varphi_0$ which is a $\omega$-psh function with zero Lelong number at all point: i.e $\f_t\in C^\infty((0,T]\times X)$ satisfies
$$\frac{\partial \f_t}{\partial t}\leq \log\frac{\left(\theta_t+dd^c \f_t\right)^n}{\Omega}-F(t,z,\f_t),$$
 and $\varphi_t\rightarrow \varphi_0$ in $L^1(X)$. Suppose that $\varphi_0\leq \psi_0$, then $\varphi_t\leq \psi_t$.
\end{prop}
\begin{proof}
Fix $\epsilon>0$ and note that $\varphi-\psi$ is a smooth function on $[\epsilon,T]\times X$. It follows from Proposition \ref{comparison} that $$\varphi-\psi\leq e^{\lambda T}\max\,\bigg\{\sup_X(\f_\epsilon-\psi_\epsilon); 0\bigg\}.$$
We have $(\epsilon,x)\mapsto \psi_\epsilon(x)$ is smooth by assumption. Using Hartogs' Lemma and the fact that $\varphi_t$ converges to $\varphi_0$ in $L^1(X)$ as $t\rightarrow 0$, we get 
$$\sup_X(\varphi_\epsilon-\psi_\epsilon)\rightarrow \sup_X(\varphi_0-\psi_0)\leq 0\,\quad \text{as}\, \, \epsilon\rightarrow 0.$$
This implies that $\varphi_t\leq\psi_t$ for all $0\leq t \leq T$. \end{proof}
\subsection{Evans-Krylov and Schauder estimates for  Monge-Amp\`ere flow}
The Evans-Krylov and Schauder theorems for nonlinear  elliptic equations 
$$F(D^2u)=f$$
with $F$ concave, have used to show that bounds on $u,D^2u$ imply $C^{2,\alpha}$  on $u$ for some $\alpha>0$ and  higher order bounds on $u$. There are also Evans-Krylov estimates for parabolic equations (see \cite{Lie}), but the precise version which we need is as follows
\begin{thm}\label{Evan-Krylov 2}
Let $U\Subset\C^n$ be an open subset and $T\in (0,+\infty)$. Suppose that $u\in C^\infty ([0,T]\times\bar{U})$ and $(t,x,s)\mapsto f(t,x,s)$ is a function in $C^\infty ([0,T]\times\bar{U}\times \R)$, satisfy 
\begin{equation}\label{local parabolic equation}
\frac{\partial u}{\partial t}=\log \det \left( \frac{\partial^2 u}{\partial z_j\partial \bar{z}_k} \right)+f(t,x,u).
\end{equation} 
In addition, assume that there is a constant $C>0$ such that 
$$
\sup_{(0,T)\times U}\left(|u|+\left| \frac{\partial u}{\partial t}\right|+|\nabla u|+|\Delta u| \right)\leq C.$$
Then for any compact $K\Subset U$, for each $\e>0$ and $p\in \N$,
$$||u||_{C^p([\e,T]\times K)}\leq C_0.$$
where $C_0$ only depends on $C$ and $||f||_{C^q([0,T]\times \bar{U}\times [-C,C])}$ for some $q\geq p-2$.
\end{thm}
The proof of this theorem follows the arguments of Boucksom-Guedj 
\cite[Theorem 4.1.4]{BG13} where the function $f$ is independent of $u$.

\medskip

First of all, we recall the {\sl parabolic $\alpha$-H\"older norm} of a function $f$ on the cylinder $Q=U\times (0,T)$:
$$||f||_{C^\alpha(Q)}:=||f||_{C^0(Q)}+[f]_{\alpha,Q},$$
where $$[f]_{\alpha, Q}:=\sup_{X,Y\in Q, X\neq Y}\frac{|f(X)-f(Y)|}{\rho^\alpha(X,Y)}$$
is the {\sl $\alpha$-H\"older seminorm with respect to} the {\sl parabolic distance}
$$\rho\big((x,t),(x',t')\big)=|x-x'|+|t-t'|^{1/2}.$$ 
For each $k\in \N$, the {\sl $C^{k,\alpha}$-norm} is defined as
$$||f||_{C^{k,\alpha}(Q)}:=\sum_{|m|+2j\leq k}||D^m_xD^j_tf||_{C^\alpha(Q)}.$$ 
If $(\omega_t)_{t\in (0,T)}$ is a path of differential forms on $U$, we can similarly define $[\omega_t]_{\alpha,Q}$ and $||\omega_t||_{C^{k,\alpha}(Q)}$, with respect to the flat metric $\omega_U$ on $U$.
\medskip

The first ingredient in the proof of Theorem \ref{Evan-Krylov 2} is the Schauder estimates for linear parabolic equations.
\begin{lem}\label{second order schauder}$($\cite[Theorem 8.11.1]{Kry},\cite[Theorem 4.9]{Lie}$)$ 
Let $(\omega_t)_{t\in(0,T)}$ be a smooth path of K\"ahler metrics on $U$ and $\omega_U$ be the flat metric on $U$. Define $Q=U\times (0,T)$, and assume that $u,g\in C^\infty(Q)$ satisfy
$$\left(\frac{\partial }{\partial t}-\Delta_t  -c(t,x)\right)u(t,x)=g(t,x),$$
where $\Delta_t$ is the Laplacian with respect to $\omega_t$. Suppose also that there exist $C>0$ and $0<\alpha<1$ such that on $Q$ we have 
$$C^{-1}\omega_U\leq \omega_t\leq C\omega_U ,\, ||c||_{C^\alpha(Q)}\leq C\text{ and }\, [\omega_t]_{\alpha,Q}\leq C.$$
Then for each $Q'=U'\times (\e,T)$ with $U'\Subset U$, we can find a constant $A$ only depending on $U'$, $\e$ and $C$ such that 
$$||u||_{C^{2,\alpha}(Q)}\leq A(||u||_{C^0(Q)})+||g||_{C^\alpha(Q)}).$$
\end{lem}
The second ingredient in the proof Theorem \ref{Evan-Krylov 2} is the following Evans-Krylov estimates type for complex Monge-Amp\`ere flows. 
\begin{lem}\label{Evans-Krylov 1}$($\cite[Theorem 4.1]{Gil11}$)$
Suppose $u,g\in C^\infty(Q)$ satisfy
$$\frac{\partial u}{\partial t}=\log\det
\frac{\partial^2u}{\partial z_{j}\partial\bar{z}_k}+g(t,x),$$
and assume also that there exists a constant $C>0$ such that
$$C^{-1}\leq \left( \frac{\partial^2u}{\partial z_j\partial \bar{z}_k}\right)\leq C\, \text{ and }\, \left|\frac{\partial g}{\partial t}\right|+|dd^c g|\leq C.$$
Then for each $Q'=U'\times (\e,T)$ with $U'\Subset U$ an open subset and $\e\in(0,T)$, we can find $A>0$ and $0<\alpha<1$ only depending on $U',\e$ and $C$ such that
$$[dd^cu]_{\alpha,Q'}\leq A.$$  
\end{lem}
\begin{proof}[Proof of Theorem \ref{Evan-Krylov 2}.]
In the sequel of the proof, we say that a constant is under control if it is bounded by the terms of $C,\e$ and $||f||_{C^q([0,T]\times \bar{U}\times[-C,C])}$.

\medskip
Consider  the path $\omega_t:=dd^cu_t$ of K\"ahler forms on $U$. Denote by  $\omega_U$ the flat metric on $U$. It follows from \ref{local parabolic equation} that 
$$\omega^n_t=\exp\left(\frac{\partial u}{\partial t}-f\right)\omega^n_U.$$
Since $\frac{\partial u}{\partial t}-f$ is bounded by a constant under control by the assumption, there exists a constant $C_1$ under control such that $C_1^{-1}\omega_U^n\leq \omega^n_t\leq C_1\omega^n_U$. It follows from the ssumption that $\tr_{\omega_U}\omega_t$ is bounded. Two latter inequalities imply that $C^{-1}_2\omega_U\leq \omega_t\leq C_2\omega_U$ for some $C_2>0$ under control by considering inequalities of eigenvalues. Set $g(t,x):=f(t,x,u)$. Since $C^{-1}_2\omega_U\leq \omega_t\leq C_2\omega_U$ and
$$
\sup_{(0,T)\times U}\left(|u|+\left| \frac{\partial u}{\partial t}\right|+|\nabla u|+|\Delta u| \right)\leq C,$$
 we get $ \left|\frac{\partial g}{\partial t}\right|+|dd^c g|\leq C_3$ with $C_3$ under control. Apply Lemma \ref{Evans-Krylov 1} to (\ref{local parabolic equation}), we obtain $[dd^c u]_{\alpha,Q}$ is under control for some $0<\alpha< 1$.
\medskip

Let $D$ be any first order differential operator with constant coefficients. Differentiating (\ref{local parabolic equation}), we get 
\begin{equation}\label{differentiate 1}
\left( \frac{\partial }{\partial t}-\Delta_t-\frac{\partial f}{\partial s}\right)Du=Df,
\end{equation}
with $|u|+|\nabla u|+\left|\frac{\partial u}{\partial t}\right|+|\Delta u|$ and $[dd^cu]_{\alpha,Q}$ are under control, so  $C^0$ norm of $Du$ is under control. Applying the parabolic Schauder estimates (Lemma \ref{second order schauder}) to \ref{differentiate 1} with $c(t,x)=\frac{\partial f}{\partial s}(t,x,u)$, the $C^{2,\alpha}$ norm of $Du$ is thus under control. Apply $D$ to  (\ref{differentiate 1}) we get
\begin{eqnarray*}
\left( \frac{\partial }{\partial t}-\Delta_t-\frac{\partial f}{\partial s}\right)D^2u &=& D^2f+\frac{\partial (Df)}{ds}Du+\sum_{j,k}(D\omega_t^{jk})\frac{\partial^2 Du}{\partial z_j\partial\bar{z}_k}\\
&&+\frac{\partial^2f}{\partial s^2}|Du|^2+D\left(\frac{\partial f}{\partial s}\right)Du,
\end{eqnarray*}
where the parabolic $C^\alpha$ norm of the right-hand side is under control. Thanks to the parabolic Schauder estimates \ref{second order schauder}, the $C^{2,\alpha}$ norm of $D^2u$ is under control. Iterating this procedure we complete the proof of Theorem \ref{Evan-Krylov 2}.
\end{proof}

\subsection{Monge-Amp\`ere capacity}
\begin{defi}
Let $K$ be a Borel subset of $X$. We set
$$Cap_\omega(K)=\sup\left\{\int_K MA(\f);\, \f\in PSH(X,\omega),0\leq \f\leq 1 \right\}.$$ Then we call $Cap_\omega$ is {\sl the Monge-Amp\`ere capacity with respect to $\omega
$}.
\end{defi}
\begin{defi}\label{def conv in cap}
Let $(\f_j)\in PSH(X,\omega)$. We say that $(\f_j)$ converges to $\f$ as $j\rightarrow +\infty$ {\sl in capacity} if for each $\e>0$
$$\lim_{j\rightarrow +\infty}Cap_{\omega}(|\f_j-\f|<\e)=0.$$
\end{defi}
The following Proposition \cite[Proposition 3.7]{GZ05} states that decreasing sequences of $\omega$-psh functions converge in capacity.
\begin{thm}\label{conv in cap}
Let $\f,\f_j\in PSH(X,\omega)\cap L^{\infty}(X)$ such that $(\f_j)$ decreases to $\f$, then for each $\e>0$ 
$$Cap_\omega(\left\{\f_j>\f+\e \right\})\rightarrow 0 \quad \text{ as } j \rightarrow +\infty.$$

\end{thm}
\subsection{Monge-Amp\`ere energy}\label{MA energy}
The energy of a $\omega$-psh function has been introduced in \cite{GZ07} and further studied in  \cite{BBGZ}.
For $\phi\in PSH(X,\omega)\cap L^\infty(X)$, the {\sl Aubin-Yau energy functional} is 
$$E(\phi):=\frac{1}{(n+1)V}\sum_{j=0}^n
\int_X\phi (\omega+dd^c\phi)^j\wedge \omega^{n-j},$$
where $$V:=\int_X\omega^n.$$
For any $\phi\in PSH(X,\omega)$, we set 
$$E(\phi):=\inf\big\{E(\psi);\,\psi\in  PSH(X,\omega)\cap L^\infty(X), \phi\leq \psi\big\}.$$
\begin{defi}
We say that $\phi\in PSH(X,\omega)$ has a finite energy if $E(\phi)>-\infty$ and denote by $\mathcal{E}^1(X,\omega)$ the set of all finite energy $\omega$-psh functions.
\end{defi}
Let $(\theta_t)_{t\in [0,T]}$ be a family of K\"ahler metrics on $X$ and $\Omega$ be a smooth volume form. We consider the following complex Monge-Amp\`ere flow
\begin{equation*}
(CMAF) \hskip1cm \left\{\aligned
& \dfrac{\partial \f}{\partial t}=\log\dfrac{(\theta_t+dd^c\f)^n}{\Omega}-F(t,z,\f),  \\ 
 \\
&\f(0,.)=\f_0.
\endaligned \right. 
\end{equation*}  
We set $\omega_t=\theta_t+dd^c\f_t$. 
\begin{defi}
Suppose $\f_t$ is a solution of $(CMAF)$. The {\sl energy} for $\f_t$ is
$$E(\f_t):=E_{\theta_t}(\f_t):=\frac{1}{(n+1)V}\sum_{j=0}^n
\int_X\f_t (\theta_t+dd^c\f_t)^j\wedge \theta_t^{n-j}
.$$
In particular, when $\theta_t=\omega$ for all $t\in [0,T]$ we get the Aubin-Yau energy functional.
\end{defi}
\subsection{Reduction to $\frac{\partial F}{\partial s}\geq 0$}\label{reduction}
We now consider the complex Monge-Amp\`ere flow
$$
\hskip-2cm
(CMAF) \hskip2cm \frac{\partial \f_t}{\partial t}=\log  \frac{(\theta_t+dd^c \f_t)^n}{\Omega} -F(t,z,\f),
$$
where $ F(t,z,s)\in C^\infty([0,T]\times X\times \R,\R)$ with 
$$\frac{\partial F}{\partial s}\geq -C,$$ for some $C\geq 0$. 
\medskip 

First of all, we observe that it is sufficient to prove Theorem A with $F$ satisfying $ F(t,z,s)\in C^\infty([0,T]\times X\times \R,\R)$ and $s\mapsto F(t,z,s)$ is non-decreasing. Indeed, assume that $\f_t$ is a solution of $(CMAF)$ with $\partial F/\partial s\geq-C$. By changing variables $$\phi(t,z)=e^{Bt}\varphi\big(B^{-1}(1-e^{-Bt}),z\big),$$ we get 
$$
\frac{\partial \phi_t}{\partial t}=\log  \frac{(\tilde{\theta}_t+dd^c \phi)^n}{\Omega}-\tilde{F}(t,z,\phi_t),
$$
where $\tilde{\theta}_t=e^{Bt}\theta_\frac{1-e^{-Bt}}{B}$ and $$\tilde{F}(t,z,s)=-Bs +Bnt+F\big(B^{-1}(1-e^{-Bt}), z,e^{-Bt}s\big).$$ We thus have 
$$\frac{\partial \tilde{ F} }{\partial s}= -B+\frac{\partial F}{\partial s}e^{-Bt}\geq -B-Ce^{-Bt}.$$
Choosing $B<0$ such that $-B-Ce^{-Bt}\geq 0 $ or $-Be^{Bt}\geq C$ for all $t\in [0,T]$, we get the desired equation. Note that we can not always choose $B$ for any $T>0$ because the maximal value of $-Be^{BT}$ is $1/eT$ at $B=-1/T$, but in our case we can assume $T$ is small enough such that $C< 1/eT$. 
Finally we obtain the equation 
$$
\frac{\partial \phi_t}{\partial t}=\log  \frac{(\tilde{\theta}_t+dd^c \phi)^n}{\Omega}-\tilde{F}(t,z,\phi_t),
$$
where $\phi(0,z)=\f_0$ and $\partial \tilde{F}/\partial s\geq0$.

\subsection{Strategy of the proof}
We fix $\omega$ a reference K\"ahler form. Since we are interested in the behavior near 0 of the flow, we can assume that for $0\leq t\leq T$ 
\begin{equation}\label{asume 1}
\dfrac{\omega}{2}\leq\theta_t\leq 2\omega,
\end{equation}
and there exists $\delta>0$ such that 
$$\delta^{-1}\Omega\leq \theta_t^n\leq \delta\Omega, \ \forall t\in [0,T].$$

We consider the complex Monge-Amp\`ere flow
$$
\hskip-2cm
(CMAF) \hskip2cm \frac{\partial \f_t}{\partial t}=\log  \frac{(\theta_t+dd^c \f_t)^n}{\Omega} -F(t,z,\f),
$$
where $ F(t,z,s)\in C^\infty([0,T]\times X\times \R,\R)$ is such that 
$\frac{\partial F}{\partial s}\geq -C$, for some $C\geq 0$.
Our first goal is to show the following generalization of \cite{GZ13, DiNL14}:
\begin{thm}\label{main theorem}
Let $\f_0$ be a $\omega$-psh function with zero Lelong numbers. There exists a family of smooth strictly $\theta_t-psh$ function $(\varphi_t)$ such that
$$
\frac{\partial \f_t}{\partial t}=\log  \frac{(\theta_t+dd^c \f_t)^n}{\Omega} -F(t,z,\f_t)
$$
in $(0, T]\times X,$ with $\varphi_t\rightarrow \varphi_0$ in $L^1(X),$ as $t\searrow 0^+$. This family is unique if $C=0$ and $|\frac{\partial F}{\partial t}|<C'$ for some $C'>0$. Moreover, $\varphi_t\rightarrow \varphi$ in energy if $\varphi\in {\mathcal E}^1(X,\omega)$ and $\varphi_t$ is uniformly bounded and converges to $\varphi_0$ in capacity if $\varphi_0\in L^\infty(X).$
\end{thm}
The strategy of the proof is a follows:

\begin{itemize}
\item We first reduce to the case when $\frac{\partial F}{\partial s}\geq 0$ following Section \ref{reduction}.
\medskip
\item Approximate $\f_0$ by a decreasing sequence $(\f_{0,j})$ of smooth and strictly $\omega$-psh functions by using the regularization result of Demailly \cite{Dem92,BK07}. There exists unique  solutions $\f_{t,j}\in PSH(X,\omega)\cap C^\infty(X)$ to the flow above with initial data $\f_{0,j}$.
\medskip
\item We then establish various priori estimates which will allow us to pass to the limit as $j\rightarrow \infty$. We prove for each $0<\e<T$:
\begin{enumerate}
\item $(t,z,j)\mapsto \varphi_{t,j}(z)$ is uniformly bounded on on $[\e,T]\times X\times \N$,
\medskip
\item $(t,z,j)\mapsto \dot{\varphi}_{t,j}(z)$ is uniformly bounded on $[\e,T]\times X\times\N$,
\medskip
\item $(t,z,j)\mapsto \Delta_{\omega} \varphi_{t,j}(z)$ is uniformly bounded on $[\e,T]\times X\times\N$.
\end{enumerate}
\medskip
\item Finally, we apply the Evans-Krylov theory and Schauder estimates to show that $\f_{t,j}\rightarrow \f_t$ in $C^\infty((0,T]\times X)$, as $j\rightarrow +\infty$ such that $\f_t$ satisfies $(CMAF)$. We then check that $\f_t\rightarrow \f_0$ as $t\rightarrow 0^+$,   and also study finer convergence properties:
\begin{enumerate}
\item For $\varphi_0\in L^1(X)$, we show that $\f_t\rightarrow \f_0$ in $L^1(X)$ as $t\rightarrow 0$.
\medskip
\item When $\f_0$ is bounded, we show that $\varphi_t\rightarrow \varphi_0$ in capacity. 
\medskip
\item When $\varphi_0\in\mathcal{E}^1(X,\omega)$, we show that $\varphi_{t}$ converges to $\varphi_0$ in energy as $t\rightarrow 0.$
\end{enumerate}
 \end{itemize}
\section{A priori estimates}\label{a priori estm}
In this section we prove various a priori estimates for $\f_t$ which satisfies
$$
\frac{\partial \f_t}{\partial t}=\log  \frac{(\theta_t+dd^c \f_t)^n}{\Omega} -F(t,z,\f)
$$
with a smooth strictly $\omega$-psh initial data $\varphi_0$, where $(t,z,s)\mapsto F(t,z,s)\in C^\infty([0,T]\times X\times \R,\R)$ with $\frac{\partial F}{\partial s}\geq 0$. Since we are interested in the behavior near 0 of $(CMAF)$, we can further assume that
\begin{equation}\label{condition of theta}
\theta_t-t\dot{\theta}_t\geq 0\, \text{ for }\, 0\leq t\leq T.
\end{equation}
This assumption will be used to bound the $\dot{\f}_t$ from above. 
\subsection{Bounding $\varphi_t$}
\begin{lem}\label{bound 1}
We have 
$$
\varphi_t\leq Ct+ \max\{ \sup \varphi_0,0\},
$$
where $C=-\inf_{z\in X, t\in [0,T]} F(t,x,0)+ n\log \delta$.
\end{lem}

\begin{proof}
Consider $\psi_t=Ct$, where $C=-\inf_{z\in X, t\in [0,T]} F(t,x,0)+ n\log \delta$. Thus we have
$$\log\frac{(\theta_t+dd^c\psi_t)^n}{\Omega}=\log\frac{\theta_t^n}{\Omega}\leq n\log \delta.$$
Now $F(t,z,\psi_t)\geq F(t,z,0)\geq \inf_{z\in X, t\in [0,T]} F(t,x,0)$, since we assume $s\mapsto F(.,.,s)$ is increasing. Therefore
$$\frac{\partial\psi_t}{\partial t}\geq \frac{(\theta_t+dd^c\psi_t)^n}{\Omega}-F(t,z,\psi_t).$$
Apply Proposition \ref{comparison} for $\f_t$ and $\psi_t$, we get $\varphi_t\leq Ct+ \max\{ \sup \varphi_0,0\}$. \end{proof}

We now find a lower bound of $\f_t$ which does not depend on $\inf_X\f_0$. First, we assume that $\theta_t\geq \omega+t\chi,\,\forall t\in [0,T],$ for some smooth $(1,1)$-form $\chi$. Fix $0<\beta < +\infty$ and $0<\alpha$ such that 
$$\chi+(2\beta-\alpha)\omega\geq 0.$$
It follows from Skoda's integrability theorem \cite{Sko} that $e^{-2\beta\f_0}\omega^n$ is absolutely continuous with density in $L^p$ for some $p>1$. This is where we use the crucial assumption that $\varphi_0$ has zero Lelong number at all points.  Ko\l{}odziej's uniform estimate \cite{Kol98} insures the existence of a continuous $\omega$-psh solution $u$ of the equation
\begin{equation*}
\alpha^n(\omega+dd^c u)^n=e^{\alpha u-2\beta\f_0}\omega^n.
\end{equation*}  
Assume that $\phi_t$ is solution of the following equation
\begin{equation*}
\left\{\aligned
 &\dfrac{\partial \phi_t}{\partial t} =\log\dfrac{(\omega+t\chi+dd^c\phi)^n}{\omega^n},  \\ 
 \\
 &\phi(0,.)=\f_0.
\endaligned\right. 
\end{equation*}
By Lemma 2.9 in \cite{GZ13} we have
\begin{equation}
\phi_t(z)\geq (1-2\beta t)\f_0(z)+\alpha t u(z)+n(t\log t-t).
\end{equation}
Using this we have the following lemma:
\begin{lem}\label{bound from below}
For all $z\in X $ and $t\in (0,T]$, we have
\begin{equation}
\f_t(z)\geq\phi_t+ At\geq (1-2\beta t)\f_0(z)+\alpha t u(z)+n(t\log t-t)+ At,
\end{equation}
where  $A$ depend on $\sup_X\f_0$. In particular, there exists $c(t)\geq 0$ such that
$$\f_t(z)\geq \f_0(z)-c(t),$$
 with $c(t)\searrow 0$ as $t\searrow 0$. 

\end{lem}
\begin{proof}
There is $\sigma>0$ such that  $\sigma^{-1}\omega^n\leq \Omega\leq \sigma \omega^n$, 
so we may assume that
$$\dfrac{\partial \phi_t}{\partial t}\leq \log\dfrac{(\theta_t+dd^c\phi_t)^n}{\Omega}.$$
Thanks to Lemma \ref{bound 1}, $\f_t\leq  C_0$ with $C_0>0$ depends on $\sup_X\f_0$ and $T$. As we assume $s\mapsto F(.,.,s)$ is increasing, $F(t,z,\f_t)\leq F(t,z,C_0)$.
Replacing $\f_t$ by $\f_t-At$ and $F$ by $F-A$, where $$A:=\sup_{[0,T]\times X} F(t,z,C_0),$$ we can assume that $$\sup_{[0,T]\times X} F(t,z,\sup_{[0,T]\times X} \f_t)\leq 0.$$
Hence we have
\begin{align*}
\frac{\partial \f_t}{\partial t}&=\log\dfrac{(\theta_t+dd^c\f_t)^n}{\Omega}-F(t,z,\f_t) \\
&\geq \log\dfrac{(\omega+t\chi+dd^c\f_t)^n}{\Omega}.
\end{align*}
Apply the comparison theorem (Proposition $\ref{comparison}$) for $\f_t$ and $\phi_t$ we have $\f_t\geq  \phi_t$. In general, we get  
$$\f_t(z)\geq \phi_t+At\geq (1-2\beta t)\f_0(z)+\alpha t u(z)+n(t\log t-t)+At.$$ 
\end{proof}

\subsection{Upper bound for $\dot{\f_t}$}

We now prove a crucial estimate which allows us to use the uniform version of Kolodziej's estimates in order to get the bound of $Osc_X \f_t$.
\begin{prop}\label{bound f' above}
Fix $\e\in (0,T)$. There exists $0<C=C(\sup_X \varphi_0,\e,T)$ such that for all $\e\leq t\leq T$ and $z\in X$,
$$\dot{\f}_t(z)\leq \frac{-\f_\e(z)+C}{t}\leq \frac{-\phi_\e(z)+C}{t}-A,$$
where $\phi_t$ and $A$ are as in Lemma \ref{bound from below}.
\end{prop}
\begin{proof}
We consider $G(t,z)= t\dot{\f_t}-\f_t-nt + Bt^2/2 $, with $B<\min  F'$ on $[\e,T]\times X$. We obtain
$$\frac{\partial G}{\partial t}= t\ddot{\f}_t-n= t\Delta_{\omega_t}\dot{\f} + t \tr_{\omega_t}\dot{\theta}_t-t\frac{\partial F}{\partial s}\dot{\f}-tF'-n+Bt,$$ 
and
$$\Delta_{\omega_t} G=t\Delta_{\omega_t}\dot{\f}- \Delta_{\omega_t}\varphi_t=t\Delta_{\omega_t}\dot{\f}- (n-\tr_{\omega_t}\theta_t),$$
 hence
$$\left(\frac{\partial }{\partial t}-\Delta_{\omega_t} \right)G=-t\dot{\f}\frac{\partial F}{\partial s}+t( B-F') -\tr_{\omega_t}(\theta_t-t\dot{\theta}_t).$$
Since we assume that $\theta_t-t\dot{\theta}_t\geq 0$ and $B<\min F'$, we get
$$\left(\frac{\partial }{\partial t}-\Delta_{\omega_t} \right)G< -t\dot{\f}\frac{\partial F}{\partial s}.$$
If $G$ attains its maximum at $t=\e$, we have the result. Otherwise, assume that $G$ attains its maximum at $(t_0,z_0)$ with $t_0>\e$, then at $(t_0,z_0)$ we have $$0\leq\left(\frac{\partial }{\partial t}-\Delta_{\omega_t} \right)G< -t_0\frac{\partial F}{\partial s}\dot{\f}.$$
Since $\frac{\partial F}{\partial s}\geq 0$ by the hypothesis, we obtain $\dot{\f}(t_0,z_0)<0$ and 
$$t\dot{\f_t}-\f_t-nt+Bt^2/2\leq -\f_{t_0}(z_0)-nt_0+Bt_0^2/2.$$ Using Lemma \ref{bound from below} we get $\f_{t_0}\geq \f_\e-C(\e)$, hence
$$t\dot{\f_t}\leq \f_t-\f_\e+C_1.$$
It follows from Lemma \ref{bound 1} that $\f_t\leq C_2(\sup \f_0,T)$,
so  $$\dot{\f_t}(x)\leq \dfrac{-\f_\e+C}{t},$$ 
where $C$ depends on $\sup\f_0,\e, T$.
Since $\varphi_\e\geq \phi_\e+At$ (Lemma \ref{bound from below}), we obtain the desired inequality.
\end{proof} 
\subsection{Bounding the oscillation of $\f_t$}
Once we get an upper bound for $\dot{\f_t}$ as in Proposition \ref{bound f' above}, we can bound the oscillation of $\f_t$ by using the uniform version of Kolodziej's estimates. Indeed, observe that $\f_t$ satisfies 
$$(\theta_t+dd^c\f_t)^n=H_t\Omega,$$
then by Proposition \ref{bound f' above}, for any $\e\in (0,T)$,
$$H_t=\exp(\dot{\f_t}+F)\leq \exp(\frac{-\phi_\e+C}{t}+C')$$
are uniformly in $L^2(\Omega)$ for all $t\in [\e,T]$ since $\phi_\e$ is smooth. Thanks to the uniform version of Kolodziej's estimates \cite{Kol98, EGZ08}, we infer that the oscillation of $\f_t$ is uniformly bounded:
\begin{thm}
Fix $0<t<T$. There exist $C(t)>0$ independent of $\inf_X\f_0$ such that
$$Osc_X(\f_t)\leq C(t).$$
\end{thm}
\subsection{Lower bound for $\dot{\f_t}$}
The next result is similar to \cite[Lemma 3.2]{ST} and \cite[Proposition 3.3]{GZ13}. 
\begin{prop}\label{bound f'}
Assume $\f_0$ is bounded. There exist constants $A>0$ and $C=C(A, Osc_X\f_0)>0$ such that  for all $(x,t)\in X\times (0,T]$, 
$$\dot{\f}\geq n\log t-AOsc_X \f_0-C,$$

\end{prop}
\begin{proof}
We consider $H(t,x)=\dot{\f_t}+A\f_t-\alpha(t)$, where $\alpha\in C^\infty(\R^+,\R)$  will be chosen hereafter. We have
\begin{align*}
\frac{\partial H}{\partial t}&=\ddot{\f_t}+ A\dot{\f_t}-\dot{\alpha}\\
&=\Delta_{\omega_t}\dot{\f_t}+\tr_{\omega_t}\dot{\theta}_t- F'-\frac{\partial F}{\partial s}\dot{\f_t}+A\dot{\f_t}-\dot{\alpha},
\end{align*}
and 
$$\Delta_{\omega_t} H=\Delta_{\omega_t}\dot{\f_t}+A \Delta_{\omega_t}\f_t.$$ 
Therefore, we have 
\begin{align*}
\left( \frac{\partial}{\partial t}-\Delta_{\omega_t}\right)H &=A\dot{\f_t} + \tr_{\omega_t}\dot{\theta}_t - A\tr_{\omega_t}(\omega_t-\theta_t)-F'-\dot{\alpha}-\frac{\partial F}{\partial s}\dot{\f_t} \\
&=A\dot{\f_t}+ \tr_{\omega_t}(A\theta_t+\dot{\theta}_t) -An-F'-\dot{\alpha} -\frac{\partial F}{\partial s}\dot{\f_t}\\
&=(A-\frac{\partial F}{\partial s})\dot{\f_t} + \tr_{\omega_t}(A\theta_t+\dot{\theta}_t)- F'-\dot{\alpha}-An.
\end{align*}

Now $A\theta_t + \dot{\theta} \geq \omega$ with $A$ sufficiently large, hence
$$\tr_{\omega_t}(A\theta_t+\dot{\theta}_t)\geq \tr_{\omega_t}\omega.$$
Using the inequality
\begin{align*}
\tr_{\omega_t}(\omega)\geq n\left(\frac{\omega_t^n}{\omega^n}\right)^{-1/n}&=n \exp\left(\frac{-1}{n}( \dot{\f}+F) \right)\left( \frac{\Omega}{\omega^n}\right)^{-1/n}\\
&\geq\sigma^{-1/n} h_t^{-1/n}\exp(-\sup_{[0,T]\times X}F(t,z,C_0)/n),
\end{align*}
where $h_t=e^{\dot{\f}}$  and $C_0$ depends on $\sup_X\f_0$, we have 
$$\tr_{\omega_t}(A\theta_t+\dot{\theta}_t)\geq \frac{h_t^{-1/n}}{C}.$$
In addition, we apply the inequality $\sigma x>\log x-C_\sigma$ for all $x>0$ with $x=h_t^{-1/n}$ and $\sigma<<1$ to obtain  $\sigma h_t^{-1/n}=\sigma e^{-\dot{\f}/n}>-\dot{\f}/n-C_\sigma$. Finally, we can choose $A$ sufficient large and $\sigma>0$ such that 
$$(A-\frac{\partial F}{\partial s})\dot{\f} + \tr_{\omega_t}(A\theta_t+\dot{\theta}_t)\geq \frac{h_t^{-1/n}}{C_1}-C_1'.$$
Since $|F'|$ is bounded by some constant $C(Osc_X\f_0)>0$, we obtain 
$$\left( \frac{\partial}{\partial t}-\Delta_{\omega_t}\right)H> \frac{h_t^{-1/n}}{C_1}-\alpha'(t)-C_2,$$
where $C_2$ depends on $Osc_X\f_0$.
\medskip

We chose $\alpha$ such that $\alpha(0)=-\infty$. This insures that $H$ attains its minimum at $(t_0,z_0)$ with $t_0>0$. At $(t_0, z_0)$ we have 
$$C_1[C_2+\alpha'(t_0)]\geq h_{t_0}^{-1/n}(z_0),$$
hence
$$H(t_0,z_0)\geq A\f_{t_0}(z_0)-\{n \log[C_2+\alpha'(t_0)]+ \alpha(t_0) \}.$$
From Lemma \ref{bound 1} we have $\f_{t_0}\leq \sup_X \f_0+ C'$ we have
$$\dot{\f}\geq \alpha(t) - AOsc_X\f_0 -C_3-\{n\log[C_2+\alpha'(t_0)] +\alpha(t_0)\}.$$
Choosing $\alpha(t)=n\log t$ we have $$n\log [C_2+\alpha']+\alpha\leq C_4,$$ so obtain the inequality.
\end{proof}
\subsection{Bounding the gradient of $\f$}
In this section we bound the gradient of $\f$ using the same technique as in \cite[Lemma 4]{SzTo} (which is a parabolic version of B\l{}ocki's estimate \cite{Blo09}). In these articles $\theta_t=\omega$ is independent of $t$. We note that if one is interested in the special case of (twisted) K\"ahler-Ricci flows, then the gradient estimate is not needed.

\begin{prop}\label{bound grad}
Fix $\e\in [0,T]$. There exists $C>0$ depending on $\sup_X \f_0$ and $\e$ such that for all $\e\leq t\leq T$
$$|\nabla \f(z)|^2_\omega< e^{C/t}.$$
\end{prop}
\begin{proof}
Define 
$$K=t\log|\nabla \f|^2_\omega -\gamma\circ \f=t\log \beta -\gamma\circ \f,$$
where $\beta=|\nabla\f|^2_\omega$ and $\gamma\in C^\infty(\R,\R)$ will be chosen hereafter.
\medskip

If $K$ attains its maximum for $t=\e$, $\beta$ is bounded in terms of $\sup_X\f_0$ and $\e$, since $|\f_t|$ is bounded by a constant depending on $\sup_X\f_0$ and $\e$ for all $t\in[\e,T]$ (Lemma \ref{bound 1} and Lemma \ref{bound from below}).
\medskip

We now assume that $K$ attains its maximum at $(t_0,z_0)$ in $[\e,T]\times X$ with $t_0>\e$. Near $z_0$ we have $\omega=dd^c g$ for some smooth strongly plurisubharmonic $g$ and  $\theta_t=dd^ch_t$ for some smooth function $h_t$, hence $u:=h_t+\f$ is plurisubharmonic near $(t_0,z_0)$. We take normal coordinates for $\omega$ at $z_0$ such that 
\begin{align}\label{1}
&g_{i\bar{k}}(z_0)=\delta_{jk}\\  \label{2}
&g_{i\bar{k}l}(z_0)=0\\ \label{3}
&u_{p\bar{q}}(t_0,z_0)\ \text{is diagonal},
\end{align}
here we denote $\alpha_{j\bar{k}}:=\frac{\partial^2 \alpha}{\partial z_j\partial \bar{z}_k}$.

\medskip We now compute $K_p,K_{pp}$ at $(t_0,z_0)$ in order to use the maximum principle.  At $(t_0,z_0)$ we have $K_p=0$ hence
\begin{equation}\label{4}
t\beta_p=\beta \gamma'\f_p
\end{equation}  
or  
$$(\frac{\beta_p}{\beta})^2=\frac{1}{t^2}(\gamma')^2|\f_p|^2.  $$
Therefore,
\begin{align*}
K_{p\bar{p}}&= t\frac{\beta_{p\bar{p}}\beta-|\beta_p|^2}{\beta^2}-\gamma'' |\f_p|^2-\gamma'\f_{p\bar{p}}\\
&=t\frac{\beta_{p\bar{p}}}{\beta}-[t^{-1}(\gamma')^2+\gamma'']|\f_p|^2-\gamma'\f_{p\bar{p}}.
\end{align*}
\medskip Now we compute $\beta_p, \beta_{p\bar{p}}$ at $(t_0,z_0)$ with $\beta=g^{j\bar{k}}\f_j\f_{\bar{k}}$ where $(g^{j\bar{k}})=[(g_{j\bar{k}})^t]^{-1}$. We have
$$\beta_p=g^{j\bar{k}}_p\f_j\f_{\bar{k}}+g^{j\bar{k}}\f_{jp}\f_{\bar{k}}+g^{j\bar{k}}\f_{jp}\f_{\bar{k}p}.$$
At $(t_0,z_0)$, use (\ref{1}), (\ref{2})
$$g^{j\bar{k}}_p=-g^{j\bar{l}}g_{s\bar{l}p}g^{s\bar{k}}=0,$$
hence 
\begin{equation}\label{bp}
\beta_p=\sum\f_{jp}\f_{\bar{j}}+\sum \f_{p\bar{j}}\f_j,
\end{equation}
and
$$\beta_{p\bar{p}}=g^{j\bar{k}}_{p\bar{p}}\f_j\f_{\bar{k}}+ 2Re\sum \f_{p\bar{p}j}\f_{\bar{j}}+\sum |\f_{jp}^2|+\sum |\f_{j\bar{p}}|^2.$$
Note that $$R_{i\bar{j}k\bar{l}}=-g_{i\bar{j}k\bar{l}}+g^{s\bar{t}} g_{s\bar{j}k}g_{i\bar{t}\bar{l}},$$
hence, at $(t_0,z_0)$ $g^{j\bar{k}}_{p\bar{p}}=-g_{j\bar{k}p\bar{p}}=R_{j\bar{k}p\bar{p}},$ and 
$$\beta_{p\bar{p}}=R_{j\bar{k}p\bar{p}}\f_j\f_{\bar{k}}+ 2Re\sum \f_{p\bar{p}j}\f_{\bar{j}}+\sum |\f_{jp}^2|+\sum |\f_{j\bar{p}}|^2.$$
\medskip Now 
$$
\Delta_{\omega_{t_{0}}}K=\sum_{p=1}^n \frac{K_{p\bar{p}}}{u_{p\bar{p}}},$$
hence 
\begin{align*}
\Delta_{\omega_{t_{0}}}K =&t\sum\frac{R_{i\bar{k}p\bar{p}}\f_j\f_{\bar{k}}}{\beta u_{p\bar{p}}}+2tRe\sum \frac{\f_{p\bar{p}j}\f_{\bar{j}}}{\beta u_{p\bar{p}}}+t\frac{\sum |\f_{jp}|^2+|\f_{j\bar{p}}|^2}{\beta u_{p\bar{p}}}\\
&- \frac{[t^{-1}(\gamma')^2+
\gamma'']|\f_p|^2}{u_{p\bar{p}}}-\frac{\gamma'\f_{p\bar{p}}}{u_{p\bar{p}}}.
\end{align*}
Since $u_{p\bar{p}}=\f_{p\bar{p}}+h_{p\bar{p}}$ near $(t_0,z_0)$, then at $(t_0,z_0)$
$$\sum \frac{\gamma'\f_{p\bar{p}}}{u_{p\bar{p}}}=n\gamma'-\sum \frac{\gamma'h_{p\bar{p}}}{u_{p\bar{p}}}.$$
Moreover, assume that the holomorphic bisectional curvature of $\omega$ is bounded by a constant $B\in \R$ on X, then at $(t_0,z_0)$
$$t\sum\frac{R_{i\bar{k}p\bar{p}}\f_j\f_{\bar{k}}}{\beta u_{p\bar{p}}}\geq -Bt\sum \frac{1}{u_{p\bar{p}}},$$
therefore
\begin{align*}
\Delta_{\omega_{t_{0}}}K\geq &(\gamma'-tB)\sum_p \frac{1}{u_{p\bar{p}}}+2tRe\sum_{j,p} \frac{\f_{p\bar{p}j}\f_{\bar{j}}}{\beta u_{p\bar{p}}}\\
&+\frac{t}{\beta}\sum_{j,p}\frac{ |\f_{jp}|^2+|\f_{j\bar{p}}|^2}{\beta u_{p\bar{p}}} - [t^{-1}(\gamma')^2+
\gamma'']\sum_p\frac{|\f_p|^2}{u_{p\bar{p}}}-n\gamma'+\gamma'\sum\frac{th_{p\bar{p}}}{u_{p\bar{p}}}.
\end{align*}
By the maximum principle, at $(t_0,z_0)$
\begin{align*}
0\leq \left(\frac{\partial}{\partial t}-\Delta_{\omega_t}\right)K
\end{align*}
hence,
\begin{align}\label{ineq 1}\nonumber
0\leq &\log \beta-\gamma'\dot{\f}-(\gamma'-tB)\sum_p \frac{1}{u_{p\bar{p}}}+t\frac{\beta'}{\beta}-2tRe\sum_{j,p} \frac{\f_{p\bar{p}j}\f_{\bar{j}}}{\beta u_{p\bar{p}}}\\
&-\frac{t}{\beta}\sum_{j,p}\frac{ |\f_{jp}|^2+|\f_{j\bar{p}}|^2}{\beta u_{p\bar{p}}} + [t^{-1}(\gamma')^2+
\gamma'']\sum_p\frac{|\f_p|^2}{u_{p\bar{p}}}+n\gamma'.
\end{align}
We will simplify \ref{ineq 1} to get a  bound for $\beta$ at $(t_0,z_0)$. We now estimate 
$$ t\frac{\beta'}{\beta}-2tRe\sum_{j,p} \frac{\f_{p\bar{p}j}\f_{\bar{j}}}{\beta u_{p\bar{p}}}\quad \text{and }  -\frac{t}{\beta}\sum_{j,p}\frac{ |\f_{jp}|^2}{\beta u_{p\bar{p}}} + t^{-1}(\gamma')^2\sum_p\frac{|\f_p|^2}{u_{p\bar{p}}}.$$
\medskip
For the first one, we note that near $(t_0,z_0)$
$$\log\det(u_{p\bar{q}})=\dot{
\f}+F(t,z,\f)+ \log\Omega,$$
hence using 
$$\frac{d}{ds}\det A=A^{\bar{j}i}\left( \frac{d}{ds}A_{i\bar{j}} \right) \det A$$
we have at $(t_0,z_0)$
$$u^{p\bar{p}}u_{p\bar{p}j}=\frac{u_{p\bar{p}j}}{u_{p\bar{p}}}=(\dot{\f}+F(t,z,\f) +\log \Omega )_j.$$
Therefore
\begin{eqnarray*}
2tRe\sum_{j,p} \frac{\f_{p\bar{p}j}\f_{\bar{j}}}{\beta u_{p\bar{p}}}&=&2tRe\sum_{j,p} \frac{(u_{p\bar{p}j}-h_{p\bar{p}j})\f_{\bar{j}}}{\beta u_{p\bar{p}}}\\
&=&\frac{2t}{\beta}Re\sum \large( \dot{\f}+F(t,z,\f) +\log\Omega\large)_j\f_{\bar{j}}-2tRe\sum_{j,p} \frac{h_{p\bar{p}j}\f_{\bar{j}}}{\beta u_{p\bar{p}}} \\
&=& \frac{2t}{\beta} Re\sum (\dot{\f}_j\f_{\bar{j}})+\frac{2t}{\beta}Re\left( (F(t,z,\f)+\log\Omega)_j+\frac{\partial F}{\partial r}\f_j \right)\f_{\bar{j}}\\
&&-2tRe\sum_{j,p} \frac{h_{p\bar{p}j}\f_{\bar{j}}}{\beta u_{p\bar{p}}}.
\end{eqnarray*}
In addition, at $(t_0,z_0)$
\begin{align*}
t\frac{\beta'}{\beta}&=\frac{t}{\beta} \sum g^{j\bar{k}}(\dot{\f}_j\f_{\bar{k}} +\f_j\dot{\f}_{\bar{k}}) \\
&= \frac{2t}{\beta}Re(\dot{\f}_j\f_{\bar{j}}),
\end{align*}
we infer that
\begin{eqnarray*}
t\frac{\beta'}{\beta} -2tRe\sum_{j,p} \frac{u_{p\bar{p}j}\f_{\bar{j}}}{\beta u_{p\bar{p}}}&=&- \frac{2t}{\beta}Re\sum\left( F(t,z,\f) +\log\Omega\right)_j\f_{\bar{j}}-\frac{2t}{\beta}\sum \frac{\partial F}{\partial s}|\f_j|^2\\
&&+2tRe\sum_{j,p} \frac{h_{p\bar{p}j}\f_{\bar{j}}}{\beta u_{p\bar{p}}}.
\end{eqnarray*}
We may assume that $\log\beta>1$ so that 
$$\frac{|\f_{\bar{j}}|}{\beta}<C$$
By the hypothesis that $\frac{\partial F}{\partial s}\geq 0$ there exists $C_1$ depends on $\sup |\f_0|$ and $C_2$ depends on $h$ and $\e$ such that
\begin{equation}\label{5}
t\frac{\beta'}{\beta} -2tRe\sum_{j,p} \frac{u_{p\bar{p}j}\f_{\bar{j}}}{\beta u_{p\bar{p}}}<C_1t+ C_2 t\sum\frac{1}{u_{p\bar{p}}}
\end{equation}
\medskip
we now estimate
$$-\frac{t}{\beta}\sum_{j,p}\frac{ |\f_{jp}|^2}{\beta u_{p\bar{p}}} + t^{-1}(\gamma')^2\sum_p\frac{|\f_p|^2}{u_{p\bar{p}}}.$$
It follows from (\ref{4}) and (\ref{bp}) that
\begin{align*}
 &\beta_p=\sum\f_{jp}\f_{\bar{j}}+\sum \f_{p\bar{j}}\f_j,\\
 & t\beta_p=\beta \gamma'\f_p
\end{align*}
then,
$$\sum_j \f_{jp}\f_{\bar{j}}=(t^{-1}\gamma'\beta-\f_{p\bar{p}})\f_p.$$
Hence
\begin{align}\label{6}\nonumber
\frac{t}{\beta}\sum_{j,p}\frac{|\f_{jp}|^2}{u_{p\bar{p}}}&\geq \frac{t}{\beta^2}\sum_{j,p}\frac{|\sum \f_{jp}\f_{\bar{j}}|^2}{u_{p\bar{p}}}=\frac{t}{\beta^2}\sum \frac{|t^{-1}\gamma'\beta+1-u_{p\bar{p}}|^2|\f_p|^2}{u_{p\bar{p}}}\\ 
&\geq t^{-1}(\gamma')^2\sum \frac{|\f_p|^2}{u_{p\bar{p}}}-C_3\gamma' ,
\end{align}
here $C_3$ depends on $h$ and we assume $\gamma'>0$.

\medskip
We now choose 
$$\gamma(s)=As-\frac{1}{A}s^2$$
with $A$ so large that $\gamma'>0$ and $\gamma''=-2/A<0$ for all $s\leq\sup_{[0,T]\times X} \f_t $. From Lemma \ref{bound f'} we have $\dot{\f}\geq C_0+n\log t$, where $C_0$ depends on $Ocs_X \f_0$. Combining this with (\ref{ineq 1}), (\ref{5}), (\ref{6}) we obtain
%\begin{align*}
%0\leq \gamma''\sum\frac{|\f_p|^2}{u_{p\bar{p}}}-(\gamma'-Bt)\sum \frac{1}{u_{p\bar{p}}}+ \log \beta+ C_0\gamma'+C_1t,
%\end{align*}
\begin{align*}
0\leq -\frac{2}{A}\sum\frac{|\f_p|^2}{u_{p\bar{p}}}-(\gamma'-Bt-C_2t)\sum \frac{1}{u_{p\bar{p}}}+ \log \beta+ C_4\gamma'+C_1t,
\end{align*}
where $C_1,C_2,C_4$ depend on $\sup_X|\f_0|$, $h_t$ and $\e$.
If A is chosen sufficiently large, we have a constant $C_5>0$ such that 
\begin{equation}\label{ineq 2}
\sum\frac{1}{u_{p\bar{p}}} +\sum \frac{|\f_p|^2}{u_{p\bar{p}}}\leq C_5\log \beta,
\end{equation}
so we get
$(u_{p\bar{p}})^{-1}\leq C_5\log\beta$ for $1\leq p\leq n$. From (\ref{bound 1}) and (\ref{bound f' above}) we have at $(t_0,z_0)$ $$ \prod_p u_{p\bar{p}}=e^{-\dot{\f}_t+F(t,x,\f_t)}\leq C_6, $$
where $C_6$ depends on $\sup_X|\f_0|,\e$.  Then we get
$$u_{p\bar{p}}\leq C_6(C_5\log \beta)^{n-1},$$
so from (\ref{ineq 2}) we have
$$\beta=\sum |\f_p|^2\leq C_6(C_5\log \beta)^n,$$
hence $\log \beta< C_7$ at $(t_0,z_0)$. This shows that $\beta=|\nabla\f(z)|^2_\omega<e^{C/t}$ for some $C$ depending on $\sup|\f_0|$ and $\e$.
\end{proof}

\subsection{Bounding $\Delta \f_t $} 
We now use previous a priori estimates above to get a estimate of $\Delta \f$. The estimate on $|\nabla \f|^2_{\omega}$ is needed here, because $F(t,z,\f)$ depends on  $\varphi$, in contrast with \cite{GZ13,DiNL14}.
\begin{lem}\label{bound delta}
For all $z\in X$ and $s,t>0$ such that $s+t\leq T$,
$$0\leq t\log\tr_\omega(\omega_{t+s})\leq AOsc_X(\f_s)+C+[C-n\log s+ AOsc_X(\f_s)]t$$
for some uniform constants $C,A>0$. 
\end{lem}
\begin{proof}
We define
$$P=t\log\tr_{\omega}(\omega_{t+s}) -A \f_{t+s},$$
and $$u=\tr_{\omega}(\omega_{t+s})$$
with $A>0$ to be chosen latter. We set $\Delta_t:=\Delta_{\omega_{t+s}}$. Now,
\begin{align*}
\frac{\partial}{\partial t}P&=\log u+t\frac{\dot{u}}{u}-A\dot{\f}_{t+s},\\
\Delta_t P&=t\Delta_t\log u-A\Delta_t\f_{t+s}
\end{align*}
hence
\begin{equation}\label{heat operator}
\left( \frac{\partial}{\partial t}-\Delta_t \right) P=\log u+ t\frac{\dot{u}}{u} - A\dot{\f_{t+s}}-t\Delta_t \log u+A\Delta_t\f_{t+s}.
\end{equation}
First, we have 
$$A\Delta_t\f_{t+s} =An-A\tr_{\omega_{t+s}}(\theta_{t+s})\leq An-\frac{A}{2}\tr_{\omega_{t+s}}(\omega),$$
and by Proposition \ref{Laplace ineq} 
$$-t\Delta_t \log u\leq B\tr_{\omega_{t+s}}(\omega) +t\frac{\tr_{\omega}(Ric(\omega_{t+s}))}{\tr_{\omega}(\omega_{t+s})}.$$
Moreover,
\begin{align*}
\frac{t\dot{u}}{u}&=\frac{t}{u} \bigg[\Delta_\omega\left( \log \omega_{t+s}^n/\omega^n -\log\Omega/\omega^n -F(t,z,\f_{t+s}) \right) +\tr_{\omega}\dot{\theta}_t\bigg],\\
&=\frac{t}{u}\bigg[-\tr_{\omega}(Ric\ \omega_{t+s})+tr_\omega(\dot{\theta}_t+Ric\ \omega)-\Delta_\omega \left(F(t,z,\f)+\log\Omega/\omega^n\right)\bigg],
\end{align*}
with $u=\tr_{\omega}(\omega_{t+s})$, and $$\tr_{\omega_{t+s}}(\omega)\tr_{\omega}(\omega_{t+s})\geq n,$$
we get
\begin{equation}\label{ineq 3}
-t\Delta_t \log u+ \frac{t\dot{u}}{u}\leq (B+C_1)t\tr_{\omega_{t+s}}(\omega)-t
\frac{\Delta_\omega \big[F(t,z,\f)+\log\Omega/\omega^n\big]}{\tr_{\omega}(\omega_{t+s})}.
\end{equation}

Now $$\Delta_\omega F(t,z,\f_{t+s})=\Delta_\omega F(z,.)+2Re\bigg[ g^{j\bar{k}}\left(\frac{\partial F}{\partial s}\right)_{j}\f_{\bar{k}}\bigg]+\frac{\partial F}{\partial s}\Delta_\omega \f +\frac{\partial^2 F}{\partial s^2}|\nabla \f|^2_\omega.$$
So there are constants $C_2,C_3,C_4$ such that 
$$\big|\Delta_\omega\big( F(t,z,\f_{t+s})+\log\Omega/\omega^n\big)\big|\leq C_2+C_3|\nabla\f|^2_\omega +C_4\tr_\omega\omega_{t+s}.$$
Then we infer 
$$-\frac{\Delta_\omega [F(t,z,\f)+\log\Omega/\omega^n]}{\tr_{\omega}(\omega_{t+s})}\leq \frac{1}{n}\tr_{\omega_{t+s}}(\omega)(C_2+C_3|\nabla \f|^2_\omega)+C_4,$$
so from Lemma \ref{bound grad} and (\ref{ineq 3}) we have
\begin{equation}\label{ineq 4}
-t\Delta_t \log u+ \frac{t\dot{u}}{u}\leq (B+C_5)t\tr_{\omega_{t+s}}(\omega) +C_6.
\end{equation}

\medskip
From Lemma \ref{Trace ineq}  and the inequality $(n-1)\log x\leq x+C_n$, 
\begin{align*}
\log u&=\log \tr_{\omega}(\omega_{t+s})\leq \log \left( n\left( \frac{\omega_{t+s}^n}{\omega^n}\right) \tr_{\omega_{t+s}}(\omega)^{n-1} \right)\\
&=\log n + \dot{\f}_{t+s}+F(t,z,\f) +(n-1)\log \tr_{\omega_{t+s}}(\omega)\\
&\leq \dot{\f}_{t+s}+ \tr_{\omega_{t+s}}(\omega) +C_7.
\end{align*}

It follows from (\ref{heat operator}), (\ref{ineq 3}) and (\ref{ineq 4}) that
$$\left( \frac{\partial}{\partial t}-\Delta_t \right) P\leq C_8-(A-1)\dot{\f}_{t+s}+ [(B+C_5)t+1-A/2]\tr_{\omega_{s+t}}\omega. $$

\medskip We choose $A$ sufficiently large such that $(B+C_5)t+1-A/2<0$. Applying Proposition \ref{bound f'},
$$\left( \frac{\partial}{\partial t}-\Delta_t \right) P\leq C_8-(A-1)(n\log s-AOsc_X\f_s-C).$$
Now suppose $P$ attains its maximum at $(t_0,z_0)$. If $t_0=0$, we get the desired inequality. Otherwise, at $(t_0,z_0)$
$$0\leq \left( \frac{\partial}{\partial t}-\Delta_t \right) P\leq C_8-(A-1)(n\log s-AOsc_X\f_s-C).$$
Hence we get $$t\log\tr_\omega(\omega_{t+s})\leq AOsc_X(\f_s)+C+[C-n\log s+ AOsc_X(\f_s)]t.$$

\end{proof}
\begin{cor}
For all $(t,x)\in (0,T]\times X$
$$0\leq t\log \tr_\omega(\omega_{t+s})\leq 2 AOsc_X(\f_{t/2})+C'.$$
\end{cor}
\subsection{Higher order estimates}
For the higher order estimates, we can follow Sz\'ekelyhidi-Tosatti \cite{SzTo} by bounding $$S=g_\f^{i\bar{p}}g_\f^{q\bar{j}}g_\f^{k\bar{r}}\f_{i\bar{j}k}\f_{\bar{p}q\bar{r}}\, \text{ and } |Ric(\omega_t)|_{\omega_t},$$
then using the parabolic Schauder estimates in order to obtain bounds on all higher order derivatives for $\f$. Besides we can also combine previous estimates with Evans-Krylov and Schauder estimates (Theorem \ref{Evan-Krylov 2}) to get the $C^k$ estimates for all $k\geq 0$.
\begin{thm}\label{full estimates}
For each $\e>0$ and $k\in \N$, there exists $C_k(\e)$ such that
$$||\f||_{\mathcal{C}^k([\e,T]\times X)}\leq C_k(\e).$$ 
\end{thm}
\section{Proof of Theorem A}\label{proof}
\subsection{Convergence in $L^1$}\label{conv L1} We approximate $\f_0$ by a decreasing sequence $\f_{0,j}$ of smooth $\omega$-psh fuctions (using \cite{Dem92} or \cite{BK07}). Denote by $\varphi_{t,j}$ the smooth family of $\theta_t$-psh functions satisfying on $[0,T]\times X$
$$
\frac{\partial \f_t}{\partial t}=\log  \frac{(\theta_t+dd^c \f_t)^n}{\Omega} -F(t,z,\f)
$$
with initial data $\f_{0,j}$.
\medskip

It follows from the comparison principle (Proposition \ref{comparison}) that $j\mapsto \f_{j,t}$ is non-increasing. Therefore we can set
$$\f_t(z):=\lim_{j\rightarrow +\infty} \f_{t,j}(z).$$
Thanks to Lemma \ref{bound from below} the function $t\mapsto \sup_X \f_{t,j}$  is uniformly bounded, hence $\f_t$ is a well-defined $\theta_t$-psh function. Moreover, it follows from Theorem \ref{full estimates} that $\f_t$ is also smooth in $(0,T]\times X$ and satisfies
$$\frac{\partial \f_t}{\partial t}=\log  \frac{(\theta_t+dd^c \f_t)^n}{\Omega} -F(t,z,\f).$$

Observe that $(\f_t)$ is relatively compact in $L^1(X)$ as $t\rightarrow 0^+$, we now show  that $\f_t\rightarrow\f_0$ in $L^1(X)$ as $t\searrow 0^+$.
\medskip

First, let $\f_{t_k}$ is a subsequence of $(\f_t)$ such that $\f_{t_k}$ converges to some function $\psi$ in $L^1(X)$ as $t_k\rightarrow 0^+$. By the properties of plurisubharmonic functions, for all $z\in X$
$$\limsup_{t_k\rightarrow 0} \f_{t_k}(z)\leq \psi(z),$$
with equality almost everywhere. We infer that for almost every $z\in X$
$$\psi(z)=\limsup_{t_k\rightarrow 0}\f_{t_k}(z)\leq \limsup_{t_k\rightarrow 0} \f_{t_k,j}(z)=\f_{0,j}(z),$$
by continuity of $\f_{t,j}$ at $t=0$. Thus $\psi\leq \f_0$ almost everywhere.

\medskip
Moreover, it follows from Lemma \ref{bound from below} that
$$\f_t(z)\geq (1-2\beta t)\f_0(z)+\alpha t u(z)+n(t\log t-t)+ At, $$
with $u$ continuous, so
$$\f_0\leq \liminf_{t\rightarrow 0}\f_t.$$
Since $\psi\leq \f_0$ almost everywhere, we get $\psi=\f_0$ almost everywhere, so $\f_t\rightarrow \f_0$ in $L^1$.
\medskip

We next consider some cases in which the initial condition is slightly more regular.
\subsection{Uniform convergence}
If the initial condition $\f_0$ is continuous then by  Proposition \ref{comparison} we get $\f_t\in C^0([0,T]\times X)$, hence $\f_t$ uniformly converges to $\f_0$ as $t\rightarrow 0^+$. 
\subsection{Convergence in capacity}
When $\f_0$ is only bounded, we prove this convergence moreover holds in capacity (Definition \ref{def conv in cap}). It is the strongest convergence we can expect in the bounded case (cf. \cite{GZ05}). First, observe that it is sufficient to prove that $u_t:=\f_t+c(t)$ converges to $\f_0$ as $t\rightarrow 0$ in capacity, where $c(t)$ satisfies $\f_t+c(t)\geq \f_0$ as in Proposition \ref{bound from below}. Since $\f_t$ converges to $\f_0$, so does $u_t$, and we get 
$$\limsup_{t\rightarrow 0}u_t\leq \f_{0,j},$$
for all $j>0$, where $(\f_{0,j})$ is a family of smooth $\omega$-psh functions decreasing to $\f_0$ as in Section \ref{conv L1}. It follows from Hartogs' Lemma that for each $j>0$ and $\e>0$, there exists $t_j>0$ such that
$$u_t\leq \f_{0,j}+\e,\, \forall \, 0\leq t\leq t_j.$$
Therefore
$$Cap_\omega(\{u_{t}>\f_0+2\e\})\leq Cap_\omega(\{\f_{0,j}>\f_0+\e\}),$$
for all $t\leq t_j$. Since $\f_{0,j}$ converges to $\f_0$ in capacity (Proposition \ref{conv in cap}), the conclusion follows.  
\subsection{Convergence in energy} 
Using the same notations as in Section \ref{MA energy} we get the following monotonicity property of the energy.
\begin{prop}\label{monotonicity of energy}
Suppose $\f_t$ is a solution of $(CMAF)$ with initial data $\f_0\in \mathcal{E}^1(X,\omega)$. Then there exists a constant $C\geq0$ such that $t\mapsto E(\f_t)+Ct$ is increasing on $[0,T]$. 
\end{prop}
\begin{proof}
By computation we get
\begin{eqnarray*}
\frac{dE(\f_t)}{dt}=\frac{1}{V}\int_X\dot{\f_t}\omega_t +\frac{1}{(n+1)V}\sum_{j=0}^n\int_X \f_t\dot{\theta_t}\wedge[j\theta_t+(n-j)\omega_t]\wedge \omega_t^j\wedge \theta_t^{n-j-1}.
\end{eqnarray*}
For the first term, we use the concavity of the logarithm to get
$$\int_X\dot{\f_t}\omega_t^n=\int_X\log\left(\frac{\omega^n_t}{e^F\Omega}\right)\frac{\omega_t^n}{V_t}\geq -\log\left(\frac{\int_X e^{F(t,z,\f_t)}\Omega}{V_t} \right)\geq -\log (C_0\delta)$$
where $F(t,z,\f_t)\leq \log C_0$ and $$V_t:=\int_X\omega_t^n=\int_X\theta^n_t\geq \delta^{-1}V.$$

\medskip
For the second one, there is a constant $A>0$ such that $\dot{\theta_t}\leq A\theta_t$ for all $0\leq t\leq T$. We note that 
$$\int_X\f_t(\theta_t+dd^c\f_t)^j\wedge \theta_t^{n-j}\leq \int_X \f_t(\theta_t+dd^c\f_t)^{j-1}\wedge \theta_t^{n-j+1},$$
hence 
$$\frac{dE(\f_t)}{dt}\geq -C_1+C_2E(\f_t),$$
for some $C_1,C_2>0$. By Lemma \ref{bound from below} we have
$$E(\f_t)\geq C_3 E(\f_0)+C_3\geq C_4$$
Thus $t\mapsto E(\f_t)+Ct$ is increasing on $[0,T]$ for some $C>0$. 
\end{proof}

\begin{prop}
If $\f_0\in\mathcal{E}^1(X,\omega)$, then $\f_t$ converges to $\f_0$ in energy as $t\rightarrow 0$.
\end{prop}

\begin{proof}
It follows from Proposition \ref{monotonicity of energy} that $\f_t$ stays in a compact subset of the class $\mathcal{E}^1(X,\omega)$. Let $\psi=\lim_{t_k\rightarrow 0}\f_{t_k}$ be a cluster point of $(\f_t)$ as $t\rightarrow 0$. Reasoning as earlier, we have $\psi\leq \f_0$.
Since the energy $E(.)$ is upper semi-continuous for the weak $L^1$-topology (cf. \cite{GZ07}), Proposition \ref{monotonicity of energy} and the monotonicity of Aubin-Yau energy functional yield
$$E(\f_0)\leq \lim_{t_k\rightarrow 0}E(\f_{t_k})\leq E(\psi)\leq E(\f_0),$$
Therefore $E(\psi)=E(\f_0)$, so $\psi=\f_0$ and we have  $\f_t\rightarrow \f_0$ in energy.  
\end{proof}

\section{Uniqueness and stability of solution}\label{uniqueness and stability}
We now prove the uniqueness and stability for the complex Monge-Amp\`ere flow
$$
\hskip-2cm
(CMAF) \hskip2cm \frac{\partial \f_t}{\partial t}=\log  \frac{(\theta_t+dd^c \f_t)^n}{\Omega} -F(t,z,\f),
$$
where $ F(t,z,s)\in C^\infty([0,T]\times X\times \R,\R)$ with 
$$\frac{\partial F}{\partial s}\geq0\, \text{ and }\, \left|\frac{\partial F}{\partial t}\right|\leq C',$$
for some constant $C'>0$. 
\subsection{Uniqueness} 
For the uniqueness and stability of solution we follow the approach of Di Nezza-Lu \cite{DiNL14}. The author thanks Eleonora Di Nezza and  Hoang Chinh Lu for valuable discussion on the argument in \cite[Theorem 5.4]{DiNL14}.

\medskip
Suppose $\f_t$ is a solution of 
\begin{equation}\label{eq 1}
\left\{\aligned
&\dfrac{\partial \f}{\partial t}=\log\dfrac{(\theta_t+dd^c\f)^n}{\Omega}-F(t,z,\f),  \\ 
 \\
 &\f(0,.)=\f_0. 
\endaligned \right.  
\end{equation} 
Consider 
$$\phi(t,z)=e^{At}\f\left( (1-e^{-At})/A,z\right),$$
so $\phi_0=\varphi_0$.
Then 
\begin{equation}\label{eq 2}
\frac{\partial \phi_t}{\partial t}=\log  \frac{(\tilde{\theta}_t+dd^c \phi)^n}{\Omega}+A\phi_t -H(t,z,\phi_t),
\end{equation}
where $$\tilde{\theta}_t= e^{At}\theta_{\frac{1-e^{-At}}{A}},$$ and $$H(t,z,\phi)=Ant+F\big(A^{-1}(1-e^{-At}), z,e^{-At}\phi\big).$$
\medskip
Since 
$$\frac{\partial \tilde{\theta}_t}{\partial t}=Ae^{At} \theta_{\frac{1-e^{-At}}{A}}+ \dot{\theta}_{\frac{1-e^{-At}}{A}},$$
we can choose $A$ so large that $\tilde{\theta}_t$ is increasing in $t$.
Observe that the equation (\ref{eq 1}) has a unique solution if and only if the equation (\ref{eq 2}) has a unique solution.
\medskip
It follows from Lemma \ref{bound from below} that
$$\varphi\geq \varphi_0 -c(t),$$ where $c(t)\searrow 0$ as $t\searrow 0$, so for $\phi(t)$:
$$\phi\geq \phi_0-\alpha(t),$$ with $\alpha(t)\searrow 0$ as $t\searrow 0$.

\begin{thm}\label{uniqueness}
Suppose $\psi$ and $\f$ are two solutions of (\ref{eq 1}) with $\f_0\leq\psi_0$, then 
$\f_t\leq\psi_t.$ In particular, the equation (\ref{eq 1}) has a unique solution. 
\end{thm}
\begin{proof}
Thanks to the previous remark, it is sufficient to prove $u\leq v$, where $ u(t,z)=e^{At}\f\left( (1-e^{-At})/A,z\right),
$ and $v(t,z)=e^{At}\psi\left( (1-e^{-At})/A,z\right).$

\medskip
Fix $\e\in (0,T)$, define $$\tilde{v}(t,z)=v_{t+\e}+ \alpha(\e)e^{At} + n\e(e^{At}-1).$$
then $\tilde{v}_0\geq v_0=\psi_0$ and $\tilde{v}\geq v_{t+\e}$. Since we choose $A$ so large that $\tilde{\theta}_t$ is increasing, 
\begin{align*}
\frac{\partial\tilde{v}}{\partial t}&=\log\frac{(\tilde{\theta}_{t+\e}+dd^c v_{t+\e})^n}{\Omega} +A\tilde{v}_t -H(t,z,v_{t+s}) \\
&\geq \log\frac{(\tilde{\theta}_{t}+dd^c \tilde{v}_t)^n}{\Omega} +A\tilde{v}_t -H(t,z,v_{t+s})
\end{align*}
Where
\begin{eqnarray*}
H(t,z,v_{t+\e})= 2Ant-An(t+\e)+ F\bigg(A^{-1}(1-e^{-A(t+\e)})),z,e^{-A(t+\e)}v_{t+\e}\bigg).
\end{eqnarray*}
It follows from the monotonicity of $F$ in the third variable that
$$
F\bigg(A^{-1}(1-e^{-A(t+\e)}),z,e^{-A(t+\e)}v_{t+\varepsilon}\bigg)\leq F\bigg(A^{-1}(1-e^{-A(t+\varepsilon)}),z,e^{-At}\tilde{v}_t\bigg).$$
By the assumption $|\frac{\partial F}{\partial t}|<C'$, we choose $A>C'$ and get $$s\mapsto- A(t+s)+F\bigg(A^{-1}(1-e^{-A(t+s)}),z,e^{-At}\tilde{v}_t\bigg)$$ is decreasing. Thus 
$$H(t,z,v_{t+\e})\leq Ant+F\bigg(A^{-1}(1-e^{-At}),z,e^{-At}\tilde{v}_t) \bigg),$$
and
$$\frac{\partial\tilde{v}}{\partial t}\geq \log\frac{(\tilde{\theta}_{t}+dd^c \tilde{v}_t)^n}{\Omega} +A\tilde{v}_t -H(t,z,\tilde{v}_t).$$
Therefore $\tilde{v}$ is the supersolution of (\ref{eq 2}). It follows from Proposition \ref{weak comparison} that $u_t\leq \tilde{v}_t$, $\forall t\in [0,T]$. Letting $\e\rightarrow 0$, we get $u_t\leq v_t$, so $\f_t\leq \psi_t$. 
\end{proof}
\begin{rmk}
For  $\theta_t(x)=\omega(x), \Omega=\omega^n, F(t,z,s)=-2|s|^{1/2}$ and $\f_0=0$, we obtain two distinct solutions to $(CMAF)$, $\f_t(z)\equiv 0$ and $\f_t(z)=t^2$. Here $\frac{\partial F}{\partial s}$ is negative and $F$ is not smooth along $(s=0)$.
\end{rmk}
We now prove the following qualitative stability result:
\begin{thm}\label{qualitative stab}
Fix $\varepsilon>0$. Let $\varphi_{0,j}$ be a sequence of $\omega$-psh functions with zero Lelong number at all points, such that $\f_{0,j}\rightarrow \f_0$ in $L^1(X)$. Denote by $\varphi_{t,j}$ and $\f_j$ the solutions of (\ref{eq 1}) with the initial condition $\f_{0,j}$ and $\f_0$ respectively. Then
$$\f_{t,j}\rightarrow \f_{t}\  \text{ in }\  C^\infty ([\e, T]\times X)\ \text{ as }\  j\rightarrow +\infty.$$
\end{thm}
\begin{proof}
Observe that we can use previous techniques in Section \ref{a priori estm} to obtain estimates of $\f_{t,j}$ in $C^k([\e,T]\times X)$ for all $k\geq 0$. In particular, for the $C^0$ estimate, we need to have the uniform bound for 
$H_{t,j}=\exp(\dot{\f}_{t,j}+F)$ in order to use the uniform version of Kolodziej's estimates \cite{Kol98, EGZ08}. By Lemma \ref{bound f' above} we have
$$H_{t,j}=\exp(\dot{\f}_{t,j}+F)\leq \exp\bigg(\frac{-\phi_\e+C}{t}+C'\bigg),$$
where $C,C'$ depend on $\e, \sup_X {\f_{0,j}}$. Since $\f_{0,j}$ decreases to $\f_0$, we have the $\sup_X\f_{0,j}$ is uniformly bounded in term of $\sup_X\f_0$ for all $j$, so we can choose $C,C'$ to be independent of $j$. Hence there is a constant $A(t,\e)$ depending on $t$ and $\e$ such that $||H_{t,j}||_{L^2(X)}$ is uniformly bounded by $A(t,\e)$ for all $t\in [\e,T]$.
\medskip

By the Arzela-Ascoli theorem we can extract a subsequence $\f_{j_k}$ that converges to $\phi_t$ in $C^\infty([\e,T]\times X)$. Note that 
$$\dfrac{\partial \phi_t}{\partial t}=\log\dfrac{(\theta_t+dd^c\phi_t)^n}{\Omega}-F(t,z,\phi_t).$$
We now prove $\phi_t=\f_t$. From Lemma \ref{bound from below} we  get 
$$\f_{t,j_k}\geq (1-\beta t)\f_{0,j_k}-C(t),$$
where $C(t)\searrow 0$ as $t\rightarrow 0$.  
Let $j_k\rightarrow +\infty$ we get $\phi_t\geq (1-\beta t)\f_0-C(t)$, hence
$$\liminf_{t\rightarrow 0} \phi_t\geq \f_0.$$
It follows from Theorem \ref{uniqueness} that $\phi_t\geq \f_t$. For proving $\phi_t\leq \f_t$, we consider $\psi_{0,k}=\left(\sup_{j\geq k}\f_{0,j} \right)^*$, hence  $\psi_{0,k}\searrow \f_0$ by Hartogs theorem. Denote by $\psi_{t,k}$ the solution of (\ref{eq 1}) with initial condition $\psi_{0,j}$. It follows from Theorem $\ref{uniqueness}$ that $$ \psi_{t,j} \geq \f_{t,j}.$$    
Moreover, thanks to the same arguments for proving the existence of a solution in Sections 2 and 3 by using a decreasing approximation of $\f_0$, we have that $\psi_{t,j}$ decreases to $\f_t$. Thus we infer that $\phi_t\leq \f_t$ and the proof is complete.  
\end{proof}
\subsection{Quantitative stability estimate}
In this section, we prove the following stability result when the initial condition is continuous. 
\begin{thm}\label{Quantitative stab}
If $\f,\psi\in C^\infty((0,T]\times X)$  are solutions of $(CMAF)$ with continuous initial data $\f_0$ and $\psi_0$, then 
\begin{equation}\label{stability ineq}
||\f-\psi||_{C^{k}([\e,T]\times X)}\leq C(k,\e)||\f_0-\psi_0||_{L^\infty(X)
}.
\end{equation} 
\end{thm}
\begin{proof}
{\bf Step 1.} It follows from Demailly's approximation result (cf. \cite{Dem92}) that there exist two sequences $\{\varphi_{0,j}\}, \{\psi_{0,j}\}\subset PSH(X,\omega)\cap C^\infty(X)$ such that 
$$\lim_{j\rightarrow \infty}||\varphi_{0,j}-\varphi_0 ||_{L^\infty(X)}=0\quad \text{and}\quad \lim_{j\rightarrow \infty}||\psi_{0,j}-\psi_0 ||_{L^\infty(X)}=0.$$
Denote by $\varphi_{t,j},\psi_{t,j}$ solution of $(CMAF)$ corresponding to initial data $\varphi_{0,j}, \psi_{0,j}$. Moreover, thanks to Theorem \ref{qualitative stab} we obtain
$$\lim_{j\rightarrow\infty}||\varphi_{j,k}-\varphi_t||_{C^k([\e,T]\times X)}=0 \quad \text{and}\quad \lim_{j\rightarrow\infty}||\psi_{j,k}-\psi||_{C^k([\e,T]\times X)}=0.$$
Thus it is sufficient to prove (\ref{stability ineq}) with smooth functions $\varphi_0, \psi_0$.\\
{\bf Step 2.} We now assume that $\varphi_0$ and $\psi_0$ are smooth. For each $\lambda\in [0,1]$, there is a unique solution $\f^\lambda_t\in C^\infty((0, T]\times X)$ for  the complex Monge-Amp\`ere flow
\begin{equation}\label{eq 4}
\left\{\aligned
 &\dfrac{\partial \f^\lambda}{\partial t}=\log\dfrac{(\theta_t+dd^c\f^\lambda)^n}{\Omega}-F(t,z,\f^\lambda),  \\ 
 \\
 &\f^\lambda(0,.)=(1-\lambda)\f_0+\lambda\psi_0.
\endaligned\right.  
\end{equation}
By the local existence theorem, $\f^\lambda$ depends smoothly on the parameter $\lambda$. We denote by $\Delta_t^\lambda$ the Laplacian with respect to the K\"ahler form $$\omega^\lambda:=\theta_t+dd^c\f^\lambda.$$
Observe that
$$\left(\frac{\partial}{\partial t}-\Delta_t^\lambda\right) \frac{\partial \f^\lambda}{\partial \lambda}=-\frac{\partial F}{\partial s}\frac{\partial \f^\lambda}{\partial \lambda} ,$$
so
\begin{equation}\label{eq 5}
\left(\frac{\partial}{\partial t}-\Delta_t^\lambda\right)u^\lambda_t + g_\lambda(t,z)u^\lambda_t=0,
\end{equation}
where $u^\lambda_t=\frac{\partial \f^\lambda}{\partial \lambda}$ and $g_\lambda(t,z)=\frac{\partial F}{\partial s}(t,z,\f^\lambda)\geq 0$ . Moreover
 $$\psi_t-\f_t=\int^1_0 u^\lambda d\lambda,$$
thus it is sufficient to show that $$||u^\lambda_t||_{C^k([\e,T]\times X)}\leq C(k,\e)|| u^\lambda_0||_{L^\infty(X)}=C(k,\e)||\psi_0-\f_0||_{L^\infty}.$$
{\bf Step 3.}  It follows from Theorem \ref{full estimates} that for each $k\geq 0$,
$$\|g_\lambda\|_{C^k([\e, T]\times X)}\leq C_1(k,\e)\quad \text{and}\quad ||\omega^\lambda_t||_{C^k ([\e,T]\times X)}\leq C_2(k,\e),$$ for all $\lambda\in [0,1]$. Using the parabolic Schauder estimates \cite[Theorem 8.12.1]{Kry} for the equation $(\ref{eq 5})$ we get
$$||u^\lambda_t||_{C^k([\e,T]\times X)}\leq C(k,\e)|| u^\lambda_t||_{L^{\infty}(X)}.$$
{\bf Step 4.} Proving 
$$|| u^\lambda_t||_{L^\infty(X)}\leq || u^\lambda_0||_{L^\infty(X)}.$$
Indeed, suppose that $u^\lambda$ attains its maximum at $(t_0,z_0)$. If $t_0=0$, we obtain the desired inequality. Otherwise,  by the maximum principle, at $(t_0,z_0)$
$$0\leq \left(\frac{\partial}{\partial t}-\Delta_{t}^\lambda\right)u^\lambda_{t} =- g_\lambda(t_0,z_0)u^\lambda_{t_0}.$$
Since $g_\lambda\geq 0$, we get 
$$u_t^\lambda\leq \max\bigg\{0,\max_X u_0^\lambda\bigg\}.$$
Similarly, we obtain $$u_t^\lambda\geq \min\bigg\{0,\min_X u_0^\lambda\bigg\},$$
hence $$|| u^\lambda_t||_{L^\infty(X)}\leq || u^\lambda_0||_{L^\infty(X)}.$$
Finally,
\begin{equation*}
||\f-\psi||_{C^{k}([\e,T]\times X)}\leq \int_0^1 ||u^\lambda_t||_{C^k([\e,T]\times X)} d\lambda \leq C(k,\e) ||\f_0-\psi_0||_{L^\infty(X)
}.
\end{equation*}
The proof of Theorem B is therefore complete.\end{proof}

\section{Starting from a nef class}\label{nef}
Let $(X,\omega)$ be a compact K\"ahler manifold. In \cite{GZ13}, the authors  proved that the twisted K\"ahler-Ricci flow can smooth out a positive current $T_0$ with zero Lelong numbers belonging to a nef class $\alpha_0$. At the level of potentials it satisfies the  Monge-Amp\`ere flow
 \begin{equation}
\dfrac{\partial \f_t}{\partial t}=\log\dfrac{(\theta_0+t\omega+dd^c\f_t)^n}{\omega^n},   
\end{equation}
where $\theta_0$ is a smooth differential closed $(1,1)$-form representing a nef class $\alpha_0$ and $\f_0\in PSH(X,\theta_0)$ is a $\theta_0$-psh potential for $T_0$, i.e. $T_0=\theta_0+dd^c\f_0$. We prove here this is still true for more general flows we have considered:
\begin{thm}
Let $\theta_0$ be a smooth closed $(1,1)$-form representiong a nef class $\alpha_0$ and $\f_0$ be a $\theta_0$-psh fucntion with zero Lelong number at all points. Set $\theta_t:=\theta_0+t\omega$. Then there exists a unique family $(\f_t)_{t\in (0,T]}$ of smooth $(\theta_t)$-psh functions satisfying
\begin{equation}\label{starting from nef class}
\dfrac{\partial \f_t}{\partial t}=\log\dfrac{(\theta_t+dd^c\f_t)^n}{\Omega}-F(t,z,\f_t),   
\end{equation}
such that $\f_t$ converges to $\f_0$ in $L^1$.
\end{thm}

\begin{proof}
First, observe that for $\e>0$, $\theta_0+\e\omega$ is a K\"ahler form. Thanks to Theorem A, there exists a family $\f_{t,\e}$  of $(\theta_t+\e\omega)$-psh functions satisfying 
$$\dfrac{\partial \f_{t,\e}}{\partial t}=\log\dfrac{(\theta_t+\e\omega+dd^c\f_{t,\e})^n}{\Omega}-F(t,z,\f_{t,\e})$$
with initial data $\f_0$ which is a $(\theta_0+\e\omega)$-psh function with zero Lelong numbers. 

First, we prove that $\f_{t,\e}$ is decreasing in $\e$. Indeed, for any $\e'>\e$ 
\begin{align*}
\dfrac{\partial \f_{t,\e'}}{\partial t}&=\log\dfrac{(\theta_t+\e'\omega+dd^c\f_{t,\e'})^n}{\Omega}-F(t,z,\f_{t,\e'})\\
&\geq \log\dfrac{(\theta_t+\e\omega+dd^c\f_{t,\e'})^n}{\Omega}-F(t,z,\f_{t,\e'})
\end{align*}
hence  $\f_{t,\e'}\geq \f_{t,\e}$ by the comparison principle (Proposition \ref{comparison}). Then we consider
$$\f_t:=\lim_{\e\rightarrow 0+}\searrow \f_{t,\e}.$$

We now  show that $\f_t$ is bounded below (so it is not $-\infty $). Thanks to \cite[Theorem 7.1]{GZ13}, there exist a family $(\phi_t)$ of  $(\theta_0+t\omega)$-psh functions such that 
\begin{align*}
\dfrac{\partial \phi_t}{\partial t}=\log\dfrac{(\theta_0+t\omega+dd^c\phi_t)^n}{\omega^n}
\end{align*}
There is $\sigma>0$ such that  $\sigma^{-1}\omega^n\leq \Omega\leq \sigma \omega^n$, 
so we may assume that
$$\dfrac{\partial \phi_t}{\partial t}\leq \log\dfrac{(\theta_0+t\omega+dd^c\phi_t)^n}{\Omega}.$$
Moreover, $\f_{t,\e}\leq C$, where $C$ only depends on $\sup_{X}\f_0$, hence assume that $F(t,z,\f_{t,\e})\leq A$ for all $\e$ small. Changing variables, we can assume that $F(t,z,\f_{t,\e})\leq 0$, hence
\begin{align*}
\dfrac{\partial \f_{t,\e}}{\partial t}&\geq \log\dfrac{(\theta_0+t\omega+dd^c\f_{t,\e})^n}{\Omega}.
\end{align*}
Using the comparison principle (Theorem \ref{comparison}) again, we get $ \f_{t,\e}\geq \phi_t$  for all $\e>0$ small, so $\f_t\geq \phi_t$. 

\medskip
For the essential uniform bound of $\f_t$, we use the method of Guedj-Zeriahi. For $\delta>0$, we fix $\omega_\delta$ a K\"ahler form such that $\theta_0+\delta\omega=\omega_\delta+dd^ch_\delta$ for some smooth function $h_\delta$. Our equation can be rewritten, for $t\geq \delta$
\begin{equation}\label{regular equation}
(\omega_\delta+(t- \delta)\omega+dd^c(\f_t+ h_\delta))^n=H_t\Omega
\end{equation}
where 
$$H_t=e^{\dot{\f_t}+F(t,x,\f_t)}$$
are uniformly in $L^2$, since
$$\dot{\f_t}\leq \frac{-\phi_\delta+C}{t}+C,$$
 for $t\geq \delta$ as in Lemma \ref{bound from below}. Kolodziej's estimates now yields that $\f_t+h_\delta$ is uniformly bounded for $t\geq \delta$, so is $\f_t$.
\medskip 

Now apply the arguments in Section \ref{a priori estm} to the equation (\ref{regular equation}) we obtain the bounds for the time derivative, gradient, Laplacian and higher order derivatives of $\f_t+h_\delta$ in $[\delta,T]\times X$. We  thus obtain a priori estimates for $\f_t$ which allow us get the existence of solution of (\ref{starting from nef class}) and the convergence to the initial convergence in $L^1(X)$.   
\end{proof}

%\bibliographystyle{myalpha}
%\bibliographystyle{smfalpha}
%\bibliographystyle{alphaurl} % URL
%\bibliography{biblio}
%\printbibliography

\end{document}